\documentclass[a4paper,11pt,reqno]{amsart}

\usepackage{url}
\usepackage[colorlinks=true, linkcolor=blue, citecolor=ForestGreen, urlcolor=blue]{hyperref}
\usepackage{cite}

\usepackage{latexsym,amsmath,amssymb,amsfonts,amsthm}
\usepackage{mathrsfs}
\usepackage{mathtools}
\usepackage{dsfont}
\usepackage{bbm}
\usepackage{soul}     
\usepackage[utf8]{inputenc}
\usepackage{comment}

\usepackage[usenames, dvipsnames]{color}
\usepackage[usenames, dvipsnames, cmyk]{xcolor}
\usepackage{graphicx}
\usepackage[scriptsize,hang,raggedright]{subfigure}
\usepackage{multirow}
\usepackage{enumerate}
\usepackage{enumitem}

\usepackage{a4wide}
\usepackage{epsfig,epstopdf}

\numberwithin{equation}{section}

\definecolor{grey}{rgb}{.7,.7,.7}
\definecolor{refkey}{gray}{.45}
\definecolor{labelkey}{gray}{.45}

\newtheorem{theorem}{Theorem}
\newtheorem{prop}[theorem]{Proposition}
\newtheorem{lemma}[theorem]{Lemma}

\theoremstyle{remark}
\newtheorem{remark}[theorem]{Remark}
\theoremstyle{definition}
\newtheorem{defin}[theorem]{Definition}

\def\eps{\varepsilon}

\def\L{\mathscr{L}}
\def\E{\mathcal E}

\def\P{\mathcal P}
\def\R{\mathbb R}
\def\RR{\mathcal R}

\def\N{\mathbb N}
\def\Z{\mathbb Z}

\def\bal{\begin{aligned}}
\def\eal{\end{aligned}}
\def\proofof#1{\begin{proof}[Proof of #1]}

\def\XXint#1#2#3{{\setbox0=\hbox{$#1{#2#3}{\int}$} \vcenter{\vspace{-1pt}\hbox{$#2#3$}}\kern-.5\wd0}}
\def\Xint#1{\mathchoice {\XXint\displaystyle\textstyle{#1}}{\XXint\textstyle\scriptstyle{#1}}{\XXint\scriptstyle\scriptscriptstyle{#1}}{\XXint\scriptscriptstyle\scriptscriptstyle{#1}}\!\int}
\def\XXiint#1#2#3{{\setbox0=\hbox{$#1{#2#3}{\iint}$} \vcenter{\vspace{-1pt}\hbox{$#2#3$}}\kern-0.5\wd0}}

\def\intmed{\Xint{\hbox{---}}}

\def\comp{\subset\subset}
\def\spt{{\rm spt}}
\newcommand{\norm}[1]{\left\lVert#1\right\rVert}
\newcommand{\weakstar}{\overset{\ast}{\rightharpoonup}}
\DeclareMathOperator{\diam}{diam}

\newcommand{\res}{\mathop{\hbox{\vrule height 7pt width .5pt depth 0pt
\vrule height .5pt width 6pt depth 0pt}}\nolimits}

\DeclareMathOperator{\convexEnv}{co}

\newcommand{\defeq}{\coloneqq}
\newcommand{\dd}{\,\mathrm{d}}

\allowdisplaybreaks

\title[Particle approximation]{Particle approximation of nonlocal interaction energies}

\author{Davide Carazzato}
\address{Davide Carazzato\\Faculty of Mathematics, University of Vienna, Austria}
\email{davide.carazzato@univie.ac.at}

\author{Aldo Pratelli}
\address{Aldo Pratelli\\Department of Mathematics, University of Pisa, Italy}
\email{aldo.pratelli@unipi.it}

\author{Ihsan Topaloglu}
\address{Ihsan Topaloglu\\Department of Mathematics and Applied Mathematics, Virginia Commonwealth University, Richmond VA, United States}
\email{iatopaloglu@vcu.edu}

\date{\today}

\thanks{This is a post-peer-review, pre-copyedit version of an article published in Nonlinear Analysis. The final
authenticated version is available online at: \url{https://doi.org/10.1016/j.na.2025.113974}.}

\begin{document}

\begin{abstract}
    We consider Riesz-type nonlocal energies with general interaction kernels and their discretizations related to particle systems. We prove that the discretized energies $\Gamma$-converge in the weak-$*$ topology to the Riesz functional defined over the space of probability measures. We also address the minimization problem for the discretized energies, and prove the existence of minimal configurations of particles in a very general and natural setting.
\end{abstract}

\maketitle

\section{Introduction}\label{sec:intro}

In this note we consider $n$-particle interaction energies of the form
\begin{equation} \label{eq:discrete_energy}
\E_n(x_1,\ldots,x_n) = \frac{1}{n^2} \sum_{1\leq i\neq j\leq n} g(x_i-x_j)\,,
\end{equation}
where $g\colon\R^N \to \R$ is a pairwise interaction kernel. These energies are directly related to, and can be considered as, discrete versions of continuous interaction energies 	
\begin{equation} \label{eq:continuum_energy}
\E(\mu) = \int_{\R^N}\int_{\R^N} g(x-y) \dd \mu(x)\dd \mu(y)
\end{equation}
defined over probability measures $\mu\in\P(\R^N)$. In their discrete or continuum form, such pairwise interaction energies appear in many biological or physical applications, ranging from swarming models to models of molecular structure (see e.g.~\cite{BaiCarGom-Cas2024pp,CR-L2024,CarCraYao2019,F2023,Ser2024pp}, and references therein).

Some natural questions, raised by Ca\~{n}izo and Ramos-Lora in~\cite{CR-L2024} (see Section 4), are whether there exists a nontrivial measure $\mu\in\P(\R^N)$ and a sequence of minimizers $(x_1^n,\ldots,x_n^n)$ of $\E_n$ such that $\pi_n\defeq\frac 1n\,\sum_{i=1}^n \delta_{x_i^n} \weakstar \mu$ as $n\to \infty$; whether the convergence is true for all sequences of minimizers (up to rigid motions); and whether, in this case, $\mu$ is a minimizer of the continuum energy $\E$. As the authors point out, these questions are also important to justify the use of discrete models in numerical computations investigating steady states of the continuum energies. Some positive answers to these questions appear in~\cite{CarChiHua2014} and~\cite{CanPat2018}. Carrillo, Chipot, and Huang~\cite{CarChiHua2014} consider attractive-repulsive interaction kernels in the power-law form
\[
g(x) = \frac{|x|^{\beta}}{\beta} - \frac{|x|^{\alpha}}{\alpha}
\]
with $\beta>\alpha\geq 1$, whereas Ca\~{n}izo and Patacchini~\cite{CanPat2018} consider general kernels which require some regularity conditions at the origin, and in particular include power-law kernels satisfying $\alpha>2-N$. In~\cite{CR-L2024} it is conjectured that the answer to this question should also be positive when $-N<\alpha\leq 2-N$. In~\cite[Proposition~2.8]{Ser2015} and in~\cite[Proposition~3.5]{Ser2024pp}, the author uses some robust techniques to address these questions in the presence of a confining external potential. Our result, contained in Theorem~\ref{thm:gamma-convergence}, is analogous to those ones; however, the techniques we use are different, and we do not necessarily require that the interaction kernels are attractive in long distances. We also mention that these questions exhibit connections with the theory of dislocations in materials science, and some discrete-to-continuum $\Gamma$-limit results in the literature have similar flavor compared to our setting (see for example~\cite[Theorem~1.1]{MPS2017}).

Independent to the connection of convergence of minimizers, it is also an interesting endeavor to study the approximation of $\E_n(\pi_n)$ and to obtain an asymptotic expansion in terms of $n$. Here, with a slight abuse of notation, given the atomic measure $\pi_n=\frac{1}{n}\sum_{i=1}^n \delta_{x_i}$, we write $\E_n(\pi_n)$ to denote the expression in~\eqref{eq:discrete_energy}. Since there will be no confusion, it is convenient to extend $\E_n$ to the whole $\P(\R^N)$:
\[
\E_n(\mu) \coloneqq \begin{cases}
\E_n(x_1,\ldots,x_n)\ &\text{if }\mu = \frac{1}{n}\sum_{i=1}^n \delta_{x_i}, \text{ with }x_1,\ldots,x_n\in\R^N\,,\\
+\infty &\text{otherwise}\,.
\end{cases}
\]
When $g$ is given by a power-law interaction with $-N<\alpha<0$ and $\beta=2$, and when continuum minimizers are sufficiently regular, Petrache and Serfaty~\cite{PetSer2017} obtain a second-order asymptotic expansion of $\E_n(\pi_n)$ where the first-order term is $\E(\mu)$. Likewise, in~\cite{CanPat2018} Ca\~{n}izo and Patacchini obtain the $\Gamma$-convergence of discrete energies $\E_n$ to the continuum energy $\E$ in the narrow topology, again for general kernels that are not more singular than $|x|^{2-N}$ near the origin.

In this note we consider rather general interaction kernels that merely satisfy some of the following assumptions:
\begin{enumerate}[label = (H\arabic*)]\addtolength{\itemsep}{6pt}
\item\label{ass:bddbelow_lsc} $g$ is bounded from below, lower semicontinuous and is in $L^1_{\rm loc}(\R^N)$;
\item\label{ass:zerolimit} $\displaystyle \liminf_{|x|\to \infty} g(x) \geq 0$;
\item\label{ass:new} $g$ is continuous in $\R^N \setminus \{0\}$, and \emph{radial and decreasing} in a neighborhood of the origin. That is, there exists $\bar r>0$ such that whenever $0<|x|\leq |y|<\bar r$, one has $g(x)\geq g(y)$;
\item\label{ass:negativemeas} there exists a measure $\mu_0 \in \P(\R^N)$ such that $\E(\mu_0)<0$.
\end{enumerate}
For simplicity, we will always assume $g$ to be centrally symmetric, i.e. $g(x)=g(-x)$. Notice that this can always be assumed for free because the energy is unchanged if we replace $g(x)$ with $\frac{g(x)+g(-x)}{2}$. Our results are the following:

\begin{itemize}
\item If $g$ satisfies~\ref{ass:bddbelow_lsc}, then $\E_n\xrightarrow{\Gamma} \E$ in the space $\P(\R^N)$ endowed with the weak-$*$ topology (Theorem~\ref{thm:gamma-convergence}).
\item If $g$ satisfies~\ref{ass:bddbelow_lsc}, \ref{ass:zerolimit} and~\ref{ass:negativemeas}, then $\inf \E_n \to \inf \E$ as $n\to \infty$, and any limit point of a sequence $\pi_n$, for which $\lim \E_n(\pi_n)$ achieves $\liminf$ of the infimum of $\E_n$, is a minimizer of $\E$ (Proposition~\ref{prop:limit-particles}).
\item  If $g$ satisfies~\ref{ass:bddbelow_lsc}--\ref{ass:negativemeas}, then $\E_n$ admits a minimizer in the class of $n$-discrete probability measures for any $n$ sufficiently large (Theorem~\ref{thm:existence}).
\end{itemize}

We point out that the condition~\ref{ass:negativemeas} is very similar to condition~(HE) in~\cite{SST2015}, where the requirement was $\E(\mu_0)\leq 0$ instead of $\E(\mu_0)<0$. The role of this condition is crucial; indeed, in~\cite[Theorem 3.2]{SST2015} it is proved that if $g$ satisfies~\ref{ass:bddbelow_lsc} and~\ref{ass:zerolimit} then there exists a minimizer for $\E$ if and only if~(HE) holds. And of course, if minimizers for the measure problem do not exist, then the discretization of $\E$ is pointless when one considers only the ground states of $\E$.\par

Perhaps, the main novelty of our work regards the existence of the discrete minimizers, established in Theorem~\ref{thm:existence}. The so-called Morse potentials $g(x)=C_1e^{-|x|/l_1}-C_2e^{-|x|/l_2}$ are an example of kernels for which this matter is non-trivial since they are radially decreasing at large distance (see~\cite[Proposition~3.2]{CCP2015} for more details). However, we discuss in Remark~\ref{rem:trivial-existence} some situations where it is much easier to establish the existence of discrete minimizers.

The plan of the paper is as follows: In Section~\ref{sec:prelim} we state three results that will be used in the rest of the note. In Section~\ref{sec:convergence} we prove the $\Gamma$-convergence of energies, and the convergence of minimizers. Section~\ref{sec:existence} is dedicated to the existence of discrete minimizers.

\section{Preliminaries}\label{sec:prelim}

We start by collecting some definitions and notation that we will use in this note. For every natural number $n\geq1$ we define the \emph{space of $n$-discrete probability measures} as the collection of atomic probability measures that can be written as $\frac{1}{n}\sum_{i=1}^n \delta_{x_i}$, where $x_1,\ldots,x_n\in\R^N$. This space is a subset of a $nN$-dimensional vector space in the space of measures $\mathcal{M}(\R^N)$; hence, the weak-$*$ topology induced by the inclusion in $\P(\R^N)$ is equivalent to the standard topology in $(\R^N)^n$. Since the interaction kernel might have a singularity at zero (such as the Coulomb kernel), for the $n$-discrete probability measure $\pi_n=\frac{1}{n}\sum_{i=1}^n \delta_{x_i}$ it is better to consider the discrete counterpart $\E_n(\pi_n)=\E_n(x_1,\ldots,x_n)$ instead of $\E(\pi_n)$.

We define the interaction of two measures by $\E(\mu,\nu)=\iint g(x-y)\dd \mu(x)\dd \nu(y)$, and for simplicity $\E(\mu)=\E(\mu,\mu)$. By $\|\mu\|$ we denote the total variation of a measure $\mu$, and by $\convexEnv(\cdot)$ the convex hull of a set. The constant $C$ in estimates may increase from line to line.

Let us first state the Euler-Lagrange conditions for the minimizers of $\E$. The proof of this lemma can be found in~\cite{BCT2018,CDM2016,CP2025}.

\begin{lemma}[Euler-Lagrange Conditions]\label{lemma:EL-E}
Suppose that $g$ satisfies~\emph{\ref{ass:bddbelow_lsc}}. If $\mu\in\P(\R^N)$ is a minimizer of $\E$, then
\[
\begin{cases}
\psi_\mu=\E(\mu) & \mu\text{-a.e.}\, ,\\
\psi_\mu\geq \E(\mu) & \L^N\text{-a.e. in }\R^N\setminus \spt(\mu)\, ,
\end{cases}
\]
where $\psi_\mu(x) = \int g(y-x)\dd \mu(y)$. Additionally, $\psi_\mu\leq \E(\mu)$ in $\spt (\mu)$ since $\psi_\mu$ is lower semicontinuous.
\end{lemma}

Next we establish the compactness of the support of any minimizer of $\E$. Although this result appears in the literature for interaction kernels that are increasing outside a large ball, here we prove it in a more general setting.

\begin{lemma}\label{lemma:compact-support-min}
Suppose that $g$ satisfies~\emph{\ref{ass:bddbelow_lsc}}, \emph{\ref{ass:zerolimit}} and~\emph{\ref{ass:negativemeas}}. Then, any mimizer $\mu\in\P(\R^N)$ of $\E$ has compact support.
\end{lemma}
\begin{proof}
We argue by contradiction, and we suppose that $\spt \mu$ is not compact. We call $\inf g = -C\in(-\infty,0)$. Thanks to~\ref{ass:negativemeas} we know that $\E(\mu)<0$, and thus $\psi_\mu= \E(\mu)<0$ almost everywhere in $\spt \mu$. By assumption~\ref{ass:zerolimit}, there exists a radius $R>0$ such that $g \geq \E(\mu)/4$ in $\R^N\setminus B_R$. Since $\mu$ is a probability measure, there exists $R'>0$ such that
\[
    \mu(\R^N\setminus B_{R'}) < \min \bigg\{ \frac 13,\,  -\frac{\E(\mu)}{4C}\bigg\}\,.
\]
For any point $x\in \spt \mu$ with $|x|>R+R'$,
\[
\begin{split}
\psi_\mu(x) &= \int_{B_{R'}}g(y-x)\dd \mu(y) + \int_{B_{R+R'}\setminus B_{R'}}g(y-x)\dd \mu(y) + \int_{\R^N\setminus B_{R+R'}} g(y-x) \dd \mu(y)\\
&\geq \frac{\E(\mu)}{4}\mu(B_{R'})-C \mu(B_{R+R'}\setminus B_{R'})+\int_{\R^N\setminus B_{R+R'}}g(y-x)\dd \mu(y)\\
&\geq \frac{\E(\mu)}{4}-C \mu(\R^N\setminus B_{R'})+\int_{\R^N\setminus B_{R+R'}}g(y-x)\dd \mu(y)\\
&\geq \frac{\E(\mu)}{2}+\int_{\R^N\setminus B_{R+R'}}g(y-x)\dd \mu(y)\,.
\end{split}
\]
Thanks to Lemma~\ref{lemma:EL-E},for $\mu$-a.e. $x\in\R^N\setminus B_{R+R'}$ we have that $\psi_\mu(x)=\E(\mu)$, and for those points we obtain that
\[
\int_{\R^N\setminus B_{R+R'}}g(y-x)\dd \mu(y) \leq \frac{\E(\mu)}{2}\,.
\]
As a consequence, considering $\tilde \mu=\mu\res (\R^N\setminus B_{R+R'})$ and $\nu=\frac{\tilde \mu}{\norm{\tilde \mu}}$ (that is well-defined since $\spt \mu$ is non-compact), we have
\[
\begin{split}
\E(\nu) &= \frac{1}{\norm{\tilde \mu}^2}\E(\tilde \mu) = \frac{1}{\norm{\tilde \mu}^2} \int_{\R^N\setminus (B_{R+R'})}\left(\int_{\R^N\setminus B_{R+R'}}g(y-x)\dd \mu(y)\right)\dd \mu(x)\\
&\leq \frac{1}{\norm{\tilde \mu}^2}\int_{\R^N\setminus B_{R+R'}}\frac{\E(\mu)}{2}\dd \mu(x) = \frac{\E(\mu)}{2\norm{\tilde \mu}}\,.
\end{split}
\]
Since $\norm{\tilde \mu}\leq \mu(\R^N\setminus B_{R'})<1/3$, we get $\E(\nu)\leq \frac{3}{2}\E(\mu)$. However, $\E(\mu)<0$, and thus  $\E(\nu)<\E(\mu)$, contradicting the minimality of $\mu$.
\end{proof}

We also recall the concentration-compactness principle for the convenience of the reader, that dates back to the works of Lions \cite{L1984,L1984-1}. This well-known lemma is an essential tool to obtain compactness of minimizing sequences in variational problems. We refer to~\cite[Lemma 4.3]{S2008} for the proof of this particular version.

\begin{lemma}[Concentration compactness]\label{lemma:concentration-compactness}
Let $\mu_n\in\P(\R^N)$ be a given sequence of probability measures. Then there exists a subsequence (not relabelled) such that one of the following holds:
\begin{enumerate}\addtolength{\itemsep}{6pt}
\item \emph{(Compactness)} There exists a sequence of points $x_n\in\R^N$ such that, for every $\eps>0$, there exists $R>0$ large enough such that $\mu_n(B_R(x_n))>1-\eps$.
\item \emph{(Vanishing)} For every $\eps>0$ and every $R>0$ there exists $\bar n\in\N$ such that
\[
\mu_n(B_R(x))<\eps\qquad \forall \, x\in\R^N, \ \forall\, n>\bar n\,.
\]
\item \emph{(Dichotomy)} There exist $\lambda \in(0,1)$ and a sequence of points $x_n\in\R^N$ with the following property: for any $\eps>0$, there exist $R>0$ and two sequences of non-negative measures $\mu^1_n$ and $\mu^2_n$ so that, for any $R'>R$, for every $n$ large enough one has
\begin{gather*}
\mu^1_n+\mu^2_n\leq \mu_n\,,\\
\spt \mu^1_n\subset B_R(x_n),\quad \spt \mu^2_n\subset \R^N\setminus B_{R'}(x_n)\,,\\
\left|\mu^1_n(\R^N)-\lambda\right|+\left|\mu^2_n(\R^N)-(1-\lambda)\right|<\eps\,.
\end{gather*}
\end{enumerate}
\end{lemma}

\bigskip

\section{Convergence of Energies and Minimizers}\label{sec:convergence}

A natural notion of convergence for functionals is that of $\Gamma$-convergence, that we recall here.

\begin{defin}[$\Gamma$-convergence]
Given a topological space $\mathbb X$ and a sequence of functionals $F_n :\mathbb X \to \overline\R$, one says that \emph{the sequence $\{F_n\}$ $\Gamma$-converges to $F:\mathbb X\to\overline \R$} if the following holds:
\begin{enumerate}
\item for any $x\in \mathbb X$ and any sequence $\{x_n\}$ such that $x_n \to x$, one has
\[
F(x) \leq \liminf F_n(x_n)\,;
\]
\item for any $x\in \mathbb X$, there exists a \emph{recovery sequence}, that is, a sequence $\{x_n\}$ with $x_n\to x$ and such that
\[
F(x) \geq \limsup F_n(x_n)\,.
\]
\end{enumerate}
\end{defin}

\medskip

The next result says that $\E_n\xrightarrow{\Gamma} \E$ in the space $\P(\R^N)$ endowed with the weak-$*$ topology.

\begin{theorem}[$\Gamma$-convergence]\label{thm:gamma-convergence}
Suppose that $g$ satisfies assumption~\emph{\ref{ass:bddbelow_lsc}}. Then $\E_n\xrightarrow{\Gamma} \E$. In particular, if $\E(\mu)<+\infty$ then any recovery sequence is necessarily given by $n$-discrete probability measures, at least for $n$ large.
\end{theorem}

\begin{proof}
Since by definition we have that $\E_n(\mu)=+\infty$ whenever $\mu$ is not an $n$-discrete probability meaure, we can restrict ourselves to consider those measures in the proof of the $\Gamma$-convergence. More precisely, the liminf inequality becomes trivial if all the measures of a sequence are not $n$-discrete, while otherwise the liminf coincides with the liminf of the subsequence made by the measures which are $n$-discrete. And concerning the limsup inequality, it is trivial if $\E(\mu)=+\infty$, while otherwise the recovery sequence must necessarily be done by $n$-discrete measures (up to a finite number, of course). We divide the proof in two steps, one for the liminf inequality, and one for the limsup one.

\medskip 

\noindent\emph{Step 1.} Let $\mu\in\P(\R^N)$ be given, and let $\{\pi_n\}_n$ be a sequence of $n$-discrete probability measures such that $\pi_n\weakstar \mu$. Let us fix $K>0$, and consider the truncated kernel $g_K = g\wedge K$, with the associated energy $\E^K$ and its discretizations $\E^K_n$. Since the kernel $g_K$ is bounded, we note that the discretized measures have finite energy $\E^K$. The truncated kernel $g_K$ is lower semicontinuous, hence standard arguments for weak-$*$ convergence imply that
\[
\E^K(\mu) \leq \liminf_{n\to+\infty}\E^K(\pi_n) \leq \liminf_{n\to+\infty}\E^K_n(\pi_n) + \frac{1}{n^2}\cdot nK = \liminf_{n\to+\infty}\E^K_n(\pi_n) \leq \liminf_{n\to+\infty}\E_n(\pi_n)\,.
\]
Here the second inequality follows from the trivial estimate $g_K(0) \leq K$, and the last one is a consequence of the bound $g\wedge K\leq g$. Since the parameter $K$ is free, by letting $K\to+\infty$, the monotone convergence theorem guarantees that $\E^K(\mu)\nearrow \E(\mu)$, concluding the $\Gamma$-$\liminf$ inequality.

\medskip

\noindent\emph{Step 2.} Let $\mu\in\P(\R^N)$ be a given probability measure. We have to find a recovery sequence; hence, obtain a sequence $\pi_n\weakstar \mu$ satisfying the $\limsup$ inequality. Of course, if $\E(\mu)=+\infty$ then there is nothing to do, because any sequence which converges to $\mu$ does the job. Therefore, without loss of generality, we suppose that $\E(\mu)<+\infty$, and we look for a sequence $\{\pi_n\}$ of $n$-disrete measures. Since we are interested in finding a sequence of converging measures, we can also assume that $n>3^N$.
    
Our strategy will consist in subdividing $\R^N$ in a family of rectangles, each of them having $\mu$-measure almost equal to $1/n$, and then to replace the measure $\mu$ in each rectangle with a single Dirac mass, chosen in a suitable way inside the corresponding rectangle. Let us be precise; we start defining $l=\lceil\sqrt[N]{n}\ \rceil$, so that $l^N$ is larger than $n$, but the ratio $l^N/n$ goes to $1$. Then, we define the numbers $-\infty=a_1\leq a_2 \leq \, \dots \, \leq a_l \leq a_{l+1}=+\infty$ as
\begin{equation}\label{defstrips}
a_h = \min \bigg\{ t\in \overline\R,\, \mu\Big((-\infty, t]\times \R^{N-1}\Big)\geq \frac {h-1}l \bigg\}\,.
\end{equation}
Then, we can subdivide $\mu=\mu^1+\mu^2+ \cdots + \mu^l$, where each measure $\mu^i$ is concentrated in the strip $[a_i,a_{i+1}]\times\R^{N-1}$ and it has mass exactly $1/l$. Repeating the same construction in each strip working with the other coordinates, we end up having subdivided $\R^N$ in $l^N$ closed rectangles $\RR_i$ (some of which unbounded), and having written $\mu=\sum_{i=1}^{l^N} \mu_i$, where each measure $\mu_i$ has mass $1/l^N$ and it is concentrated on $\RR_i$. 

We can then estimate the cross-interaction between the different parts $\mu_i$ as
\begin{equation}\label{eq:cross-interactions}
\sum_{i\neq j}\E(\mu_i,\mu_j) = \E(\mu) - \sum_{i=1}^{l^N} \E(\mu_i)
\leq \E(\mu) + \frac{|\inf g|}{(l^N)}\,.
\end{equation}
Now, the energy $\E(\mu_i,\mu_j)$ is the ``average'', in the sense of the measures $\mu_i$ and $\mu_j$, of the value of $g(z_i - z_j)$ with points $z_i\in\RR_i$ and $z_j\in\RR_j$; therefore, one can easily guess that it is possible to choose a special point $x_i$ in each rectangle $\RR_i$ in such a way that
\begin{equation}\label{choicexi}
\frac 1 {l^{2N}} \sum_{i\neq j} g(x_i-x_j) \leq \sum_{i\neq j} \E(\mu_i,\mu_j)\,.
\end{equation}
Formally speaking, we define the measure $\Theta= \mu_1\otimes \mu_2\otimes \cdots \otimes \mu_{l^N}$ on the space $(\R^N)^{l^N}$, we define the function $G\colon (\R^N)^{l^N}\to \R$ as
\[
G(z_1,\ldots,z_{l^N}) \coloneqq \sum_{i\neq j}g(z_i-z_j)\,,
\]
and we notice that
\[\begin{split}
\sum_{i\neq j} \E(\mu_i,\mu_j) &= \sum_{i\neq j} \iint g(z_i-z_j) \dd \mu_i(z_i)\dd\mu_j(z_j)
= (l^N)^{l^N-2}\,\sum_{i\neq j} \int g(z_i-z_j) \dd \Theta(z_1,\, \dots\,,\, z_{l^N})\\
&= (l^N)^{l^N-2}\, \int G(z_1,\, z_2,\, \,,\, \dots\,,\, z_{l^N}) \dd \Theta(z_1,\, \dots\,,\, z_{l^N})\,.
\end{split}\]
Then, there is a point $X=(x_1,\, x_2,\, \dots\,,\, x_{l^N})$ in the support of $\Theta$ such that
\[
G(X) \leq \intmed G(Z)\dd\Theta(Z) = (l^N)^{l^N}\int G(Z) \dd\Theta(Z)\,,
\]
which by the above estimate precisely implies~(\ref{choicexi}). In particular, the fact that $X$ belongs to the support of $\Theta$ implies that each $x_i$ belongs to the rectangle $\RR_i$.

Now, we are ready to define the $n$-discrete measure $\pi_n$. The above construction would suggest to set $\pi_n$ as the sum of the Dirac masses in the $l^N$ points $x_i$, each with mass $1/l^N$; this is not possible since $l^N$ is slightly larger than $n$, so we need a simple adjustment. More precisely, we define the $n$-discrete measure $\pi_n = \frac 1n \sum_{i=1}^n \delta_{x_i}$, so basically we ``ignore'' the last $l^N-n$ points.

Keeping in mind the definition of the energy $\E_n$, and using~(\ref{choicexi}) and~(\ref{eq:cross-interactions}), we have then
\[\begin{split}
\E_n(\pi_n)&= \frac 1 {n^2} \sum_{1\leq i\neq j\leq n} g(x_i-x_j)
\leq \frac 1 {n^2} \sum_{1\leq i\neq j\leq l^N} g(x_i-x_j) + \frac{2l^N(l^N-n)}{n^2} |\inf g|\\
&\leq \frac {l^{2N}}{n^2} \sum_{i\neq j} \E(\mu_i,\mu_j)+ \frac{2l^N(l^N-n)}{n^2} |\inf g|
\leq \frac {l^{2N}}{n^2} \,\E(\mu) + \frac{2l^N(l^N-n+1)}{n^2} |\inf g|\,.
\end{split}\]
Since by construction $l^N/n$ converges to $1$, the above estimate yields the limsup inequality
\[
\E(\mu) \geq \limsup \E_n(\pi_n)\,,
\]
and this fact will conclude the proof once we check that actually $\pi_n\weakstar \mu$.
    
In order to establish this convergence, let us fix a function $\phi\in C_c(\R^N)$, and let us call $\omega_\phi$ the corresponding modulus of continuity. Let $\eps>0$ be fixed, and let us take any $n\in\N$. Consider an index $1\leq i \leq l^N$ such that the rectangle $\RR_i$ of the above construction has all sides shorter than $\eps$. For such $i$, keeping in mind that $\mu_i$ has mass $1/l^N$, we estimate
\begin{equation}\label{goodrectangles}\begin{split}
\bigg|\int_{\RR_i} \phi(x)\dd \mu_i(x) - \frac{\phi(x_i)}n\bigg| &= \bigg|\int_{\RR_i} (\phi(x)-\phi(x_i))\dd \mu_i(x) + \phi(x_i) \bigg(\frac 1{l^N}-\frac 1n \bigg) \bigg|\\
&\leq \frac{\omega_\phi(\eps\sqrt N)}{l^N} + \sup |\phi| \bigg(\frac 1n-\frac 1{l^N} \bigg)\,.
\end{split}\end{equation}
Let now $D$ be a large constant such that the support of $\phi$ is contained in the cube $Q=[-D/2,D/2]^N$. For every $n$, we call $I_n\subseteq \{ 1,\, 2,\, \dots\, ,\, l^N\}$ the set of the indices $i$ for which the rectangle $\RR_i$ intersects the cube $Q$ and has some side larger than $\eps$. In order to estimate the number of indices in $I_n$, we note that, among the strips $[a_i,a_{i+1}]\times \R^{N-1}$ defined through~(\ref{defstrips}), at most $D/\eps+2$ can be those which intersect the cube $Q$ and whose width is larger than $\eps$. As a consequence, and doing the same estimate in all the directions, we immediately have
\begin{equation}\label{ShIn}
\# I_n \leq N\bigg(\frac D \eps + 2\bigg) l^{N-1} \,.
\end{equation}
As a consequence, keeping in mind~(\ref{goodrectangles}) we can then estimate
\[\begin{split}
\bigg|\int \phi\dd\mu &- \int \phi\dd \pi_n\bigg|=
\bigg|\sum_{i=1}^{l^N}  \int_{\RR_i}\phi\dd \mu_i - \sum_{i=1}^n \frac{\phi(x_i) }n\bigg|\\
&\leq \bigg|\sum_{i=1}^{l^N}  \int_{\RR_i}\phi\dd \mu_i - \frac{\phi(x_i) }n\bigg| + \frac{l^N-n}n\, \sup |\phi|\\
&\leq \sum_{i\notin I_n} \bigg| \int_{\RR_i}\phi\dd \mu_i - \frac{\phi(x_i) }n\bigg| + \sum_{i\in I_n} \bigg| \int_{\RR_i}\phi\dd \mu_i - \frac{\phi(x_i) }n\bigg|+ \frac{l^N-n}n\, \sup |\phi|\\
&\leq \omega_\phi(\eps\sqrt N) + \sup |\phi| \bigg(\frac {l^N-n}n \bigg)+\sup |\phi| \bigg(\frac 1{l^N} + \frac 1n\bigg)\# I_n+ \frac{l^N-n}n\, \sup |\phi|\,,
\end{split}\]
and then by~(\ref{ShIn}) we get
\[
\limsup_{n\to\infty} \bigg|\int \phi\dd\mu - \int \phi\dd \pi_n\bigg| \leq \omega_\phi(\eps\sqrt N)\,.
\]
Since this inequality holds for any $\eps>0$, the desired convergence $\pi_n\weakstar \mu$ is established, and the proof is completed.
\end{proof}

The next result relates the ground states of the discrete energies to the minimum of the continuum energy. For every $n\in\N$ we define the ground state energy and the limiting ground state energy by
\begin{align*}
m_n \defeq \inf \E_n \,, && \ell_P \defeq \liminf_{n\to+\infty}m_n\,.
\end{align*}

\medskip

\begin{prop}\label{prop:limit-particles}
Suppose that $g$ satisfies~\emph{\ref{ass:bddbelow_lsc}}, \emph{\ref{ass:zerolimit}}, and~\emph{\ref{ass:negativemeas}}.  Then
\begin{equation}\label{eq:limit-min}
\ell_P=\lim_{n\to+\infty}m_n=\inf \left\{\E(\mu) \colon \mu\in\P(\R^N)\right\}\,.
\end{equation}
Additionally, any sequence $\pi_n$ for which $\lim_{n\to+\infty}\E_n(\pi_n)$ equals $\ell_P$ is precompact in $\P(\R^N)$ with respect to the weak-$*$ convergence (up to translations), and any limit point is a minimizer of $\E$ in $\P(\R^N)$.
\end{prop}
\begin{proof}
First of all, the $\Gamma$-$\limsup$ inequality of Theorem~\ref{thm:gamma-convergence} implies that for every $\mu\in\P(\R^N)$ there is a recovery sequence $\pi_n$, and $\E(\mu)\geq \limsup \E_n(\pi_n)\geq \limsup m_n$. Since $\mu$ is generic, this implies that $\ell_P\leq \limsup m_n\leq \inf\{\E(\mu)\}\leq \E(\mu_0)<0$, where $\mu_0$ is a measure of strictly negative energy, which exists by~\ref{ass:negativemeas}.

Let now $\{\pi_{n_k}\}$ be any sequence such that $\E_{n_k}(\pi_{n_k})\to \ell_P$. We claim that the sequence is precompact in $\P(\R^N)$ with respect to the weak-$*$ topology (up to translations). If this is true, then there is a subsequence converging to some measure $\mu\in\P(\R^N)$, and the $\Gamma$-liminf inequality of Theorem~\ref{thm:gamma-convergence} gives that $\E(\mu)\leq \liminf \E_{n_k}(\pi_{n_k})=\ell_P$. Hence, the proof will be completed once we show the validity of the claim, that is, that the sequence $\pi_{n_k}$ is precompact in $\P(\R^N)$ up to translations.

Without loss of generality, we can assume that each measure $\pi_{n_k}$ is $n_k$-discrete, and concentrated on $n_k$ points. With a slight abuse of notation, we will denote these points by $x_1,\, x_2,\, \dots\,,\, x_{n_k}$ (the formally correct notation would be $x^{n_k}_1,\, x^{n_k}_2,\, \dots\,,\, x^{n_k}_{n_k}$, but since there is no risk of confusion we prefer to keep the lighter notation). Thanks to Lemma~\ref{lemma:concentration-compactness}, we only have to exclude the possibility that either vanishing or dichotomy occur.

\medskip

\noindent\emph{Excluding vanishing.} Suppose that the vanishing phenomenon in Lemma~\ref{lemma:concentration-compactness} occurs for a (not relabelled) subsequence of $\{\pi_{n_k}\}_k$. Let $\eps>0$ and take $R>0$ such that $g>-\eps$ in $\R^N\setminus B_R$. Then, there exists $\bar n\in\N$ such that $\pi_{n_k}(B_R(x))<\eps$ for every $x\in\R^N$ whenever $n_k>\bar n$. Thus, we can control the energy of $\pi_{n_k}$ as
\[\begin{split}
\E_{n_k}(\pi_{n_k})
&=\frac{1}{n_k^2}\sum_{i=1}^{n_k}\bigg( \sum_{\substack{j\neq i \\ x_j\in B_R(x_i)}} g(x_i-x_j)+\sum_{\substack{j\neq i \\ x_j\not\in B_R(x_i)}}g(x_i-x_j)\bigg)\\
&\geq \frac{1}{n_k}\sum_{i=1}^{n_k}\bigg(\big(\inf g\big) \pi_{n_k}\big(B_R(x_i)\setminus\{x_i\}\big)-\eps \pi_{n_k}\big(\R^N\setminus B_R(x_i)\big)\bigg)
\geq \eps \big(\inf g -1 \big)\,,
\end{split}\]
where the last inequality is true since $\inf g<0$. We deduce that $\liminf \E_{n_k}(\pi_{n_k})\geq 0$, which is impossible since $\E_{n_k}(\pi_{n_k})\to \ell_P<0$; thus, the vanishing phenomenon is excluded.
    
\medskip
    
\noindent\emph{Excluding dichotomy.} Suppose now that the dichotomy phenomenon in Lemma~\ref{lemma:concentration-compactness} occurs for a (nor relabelled) subsequence of $\{\pi_{n_k}\}$. In this case, the family of points $x_1,\ldots,x_{n_k}$ representing $\pi_{n_k}$ contains two subfamilies that interact very weakly. More precisely, up to reordering the points we have the family $x_1,\, \ldots,\, x_{a_k}$ with $\frac{a_k}{n_k}\to \lambda\in (0,1)$, and the family $x_{a_k+1},\, \dots\,,\, a_{a_k+b_k}$ with $\frac{b_k}{n_k}\to 1-\lambda$ so that the distance between any two points in the two different families is arbitrarily large if $k$ is large enough. Let us then call $\mu_k^1$ and $\mu_k^2$ the probability measures uniformly distributed on the two subfamilies, i.e.
\begin{align*}
\mu_k^1 = \frac{1}{a_k}\sum_{i=1}^{a_k}\delta_{x_i}\,, && \mu_k^2 = \frac{1}{b_k}\sum_{i=a_k+1}^{a_k+b_k}\delta_{x_i}\,.
\end{align*}
Since the distances between points in the two subfamilies become arbitrarily large as $k\to+\infty$, and keeping in mind~\ref{ass:zerolimit}, we have
\begin{equation}\label{eq:weak-interaction-dichotomy}
\liminf_{k\to+\infty}\min \Big\{g(x_i-x_j)\colon 1\leq i \leq a_k, \ a_k+1\leq j \leq a_k+b_k\Big\}=0\,.
\end{equation}
We can then write
\[\begin{split}
\E_{n_k}(\pi_{n_k})&=\frac{1}{n_k^2}\sum_{1\leq i\neq j\leq n_k} g(x_i-x_j)\\
&=\frac{a_k^2}{n_k^2} \,\E_{a_k}(\mu^1_k)+\frac{b_k^2}{n_k^2} \,\E_{b_k}(\mu^2_k)+\frac{2}{n_k^2}\sum_{i=1}^{a_k}\sum_{j=a_k+1}^{a_k+b_k}g(x_i-x_j)
+\frac{2}{n_k^2}\sum_{\substack{i\neq j\\ i\wedge j>a_k+b_k}} g(x_i-x_j)\,.
\end{split}\]
Now, (\ref{eq:weak-interaction-dichotomy}) implies that
\[
\liminf_{k\to+\infty} \frac{2}{n_k^2}\sum_{i=1}^{a_k}\sum_{j=a_k+1}^{a_k+b_k}g(x_i-x_j)\geq 0\,,
\]
and the fact that $(a_k+b_k)/n_k\to 1$ and that $\inf g>-\infty$ implies that
\[
\liminf_{k\to\infty} \frac{2}{n_k^2}\sum_{\substack{i\neq j\\ i\wedge j>a_k+b_k}} g(x_i-x_j)\geq 0\,.
\]
As a consequence, we obtain
\[\begin{split}
\ell_P&=\lim \E_{n_k}(\pi_{n_k})\geq \liminf \frac{a_k^2}{n_k^2}\, \E_{a_k}(\mu^1_k)+\liminf \frac{b_k^2}{n_k^2}\, \E_{b_k}(\mu^2_k)\\
&\geq \liminf \frac{a_k^2}{n_k^2}\, m_{a_k}+\liminf \frac{b_k^2}{n_k^2}\, m_{b_k}
= \ell_P \big(\lambda^2 + (1-\lambda)^2\big)\,,
\end{split}\]
which is impossible since $\ell_P<0$ and $\lambda^2+(1-\lambda)^2<1$. The contradiction excludes also the dichotomy phenomenon, and as noticed above this concludes the thesis.
\end{proof}

\begin{remark}
    We stress that the pre-compactness result in Proposition~\ref{prop:limit-particles} holds only for (asymptotically) minimal discrete measures. Under our set of assumptions we cannot expect to have pre-compactness for sequences of measures with bounded energy: if $g$ is bounded at infinity, we could arrange the particles to have uniformly bounded interaction between each other (and this phenomenon was already pointed out in \cite[Section~3]{KP2021} in a slightly different setting). Instead, if we work with an external confining potential (as in \cite{Ser2015,Ser2024pp}), or with a kernel $g$ that grows indefinitely at $\infty$ (as the kernel considered in \cite[Section~3]{F2023}), then sequences with uniformly bounded energy are pre-compact in weak-$*$ topology. Finally, we point out that in \cite[Theorem~2.4]{CKNP2023} the authors give a pre-compactness result (in a different setting) under an a-priori growth bound of the diameter of the supports of the measures.
\end{remark}

\bigskip

\section{Existence of Discrete Minimizers}\label{sec:existence}

In this last section, we are going to prove that there exist minimizers for the energy $\E_n$ for all $n$ large enough; this will be achieved with Theorem~\ref{thm:existence}. Simple examples show that the sole assumption~\ref{ass:bddbelow_lsc} does not guarantee this existence; on the contrary, it is also possible that for every large $n$ there are no minimizers. To introduce the question, we give a simple proof of a weaker result (which is not needed in the proof of Theorem~\ref{thm:existence}); namely, that with rather weak assumptions there are infinitely many indices for which a minimizer exists.

\begin{prop}\label{prop:frequent-existence}
Suppose that $g$ satisfies~\emph{\ref{ass:bddbelow_lsc}} and~\emph{\ref{ass:negativemeas}}, and that $\lim_{|x|\to\infty} g(x)=0$. Then there exists a sequence of indices $n_k\nearrow+\infty$ such that $m_{n_k}$ is a minimum.
\end{prop}
\begin{proof}
We argue by contradiction, and we assume that there exists a value $n_e\in\N$ such that $m_n$ is not attained for any $n>n_e$. Let $n>n_e$ be any fixed number, and let $\{\pi^h\}$ be a sequence of measures such that $\E_n(\pi^h) \to m_n$. Notice that there can be no minimizer for $\E_n$, thus $m_n<+\infty$, and we can assume that all measures $\pi^h$ are $n$-discrete; hence, $\pi^h$ is concentrated on the points $x^h_1,\, x^h_2,\, \dots\,,\, x^h_n$. Up to a subsequence, we can assume that for any $i,\,j\in \{1,\, 2,\, \dots\,,\, n\}$ the sequence of distances $|x^h_i-x^h_j|$ converges to a number $d_{i,j}\in [0,+\infty]$ as $h\to +\infty$. Note that by construction $d_{i,j}\leq d_{i,m}+d_{m,j}$ for any three indices $i,\,j,\, m\in \{1,\,2,\, \dots\,,\, n\}$. In particular, $\{1,\,2,\,\dots\,,\, n\}$ is subdivided in some classes of indices, in such a way that $d_{i,j}<+\infty$ if and only if $i$ and $j$ belong to the same class. Up to renumbering, we can assume that one class is $I=\{1,\, 2,\, \dots\,,\, H\}$ for some $1\leq H\leq n$. Up to a translation, we can assume that $x^h_i$ converges to a point $P_i\in\R^N$ for every $1\leq i\leq H$. Calling $\pi_I$ the $H$-discrete measure associated to the points $P_1,\, P_2,\, \dots\, ,\, P_H$, that is, $\pi_I = \frac 1 H \sum_{h=1}^H \delta_{P_h}$, we claim that
\begin{equation}\label{minopt}
m_H = \E_h (\pi_I)\,.
\end{equation}
To show this property, let us call $\pi^h_1$ (resp., $\pi^h_2$) the $H$-discrete (resp., $(n-H)$-discrete) measure associated with the first $H$ (resp., last $n-H$) points of the support of $\pi^h$, that is,
\begin{align*}
\pi^h_1 = \frac 1H\, \sum_{i=1}^H \delta_{x^h_i}\,, && \pi^h_2 = \frac 1 {n-H}\, \sum_{i=H+1}^n \delta_{x^h_i}\,,
\end{align*}
and let us also call $d^h$ the minimum of the distances $|x^h_i-x^h_j|$ for $i$ and $j$ belonging to two different classes. Note that $d^h\to \infty$ as $h\to\infty$, and then $\eta^h\to 0$, where $\eta^h = \sup \big\{ g(v)\colon |v|\geq d^h\big\}$. We can now evaluate
\begin{equation}\label{firsttry}\begin{split}
\E_n(\pi^h) &= \frac 1{n^2} \sum_{1\leq i\neq j\leq n} g(x^h_i-x^h_j)\\
&=\frac 1{n^2} \sum_{1\leq i\neq j \leq H} g(x^h_i-x^h_j) +
\frac 1{n^2} \sum_{H< i\neq j \leq n} g(x^h_i-x^h_j) +
\frac 2{n^2} \sum_{1\leq i\leq H< j \leq n} g(x^h_i-x^h_j)\\
&\geq \frac{H^2}{n^2} \, \E_H(\pi^h_1) + \frac {(n-H)^2}{n^2}\,\E_{n-H}(\pi^h_2) -\frac {2H(n-H)}{n^2}\,\eta^h \,.
\end{split}\end{equation}
Now, if~(\ref{minopt}) is false, there is a $H$-discrete measure $\tilde\pi_I$ such that $\E_H(\tilde\pi_I)<\E_H(\pi_I)$. We can then define the $n$-discrete measure
\[
\tilde\pi^h = \frac Hn\, \tilde\pi_I + \frac{n-H}n\,\tilde\pi^h_2\,,
\]
where the $(n-H)$-discrete measure $\tilde\pi^h_2$ coincides with $\pi^h_2$, possibly up to a translation in such a way that the distance between any point in the support of $\tilde\pi_I$ and any point in the support of $\tilde\pi^h_2$ is at least $d^h$. Arguing exactly as in~(\ref{firsttry}), only estimating the last term in the second line from above instead that from below, we get
\[
\E_n(\tilde\pi^h) \leq \frac{H^2}{n^2} \, \E_H(\tilde\pi_I) + \frac {(n-H)^2}{n^2}\,\E_{n-H}(\tilde\pi^h_2) +\frac {2H(n-H)}{n^2}\,\eta^h\,.
\]
Then by~(\ref{firsttry}), and keeping in mind that $\pi^h_1 \weakstar \pi_I$, we deduce
\[\begin{split}
\liminf_{h\to\infty} \E_n(\tilde\pi^h) &- \E_n(\pi^h)\leq \liminf \bigg(\frac{H^2}{n^2} \, \Big(\E_H(\tilde\pi_I) - \E_H(\pi^h_1)\Big)+\frac {4H(n-H)}{n^2}\,\eta^h\bigg)\\
&=\frac{H^2}{n^2} \liminf \Big( \E_H(\tilde\pi_I) - \E_H(\pi^h_1)\Big)
\leq \frac{H^2}{n^2} \, \Big( \E_H(\tilde\pi_I) - \E_H(\pi_I) \Big)<0\,.
\end{split}\]
This contradicts the fact that the sequence $\{\pi^h\}$ realizes the infimum in the definition of $m_H$; hence,~(\ref{minopt}) is proved. In particular, since we are assuming that the minimum is never attained for more than $n_e$ points, we deduce that $H\leq n_e$.\par

Summarizing, by the above argument we have observed that \emph{each of the classes} in which we have subdivided $\{1,\, 2,\,\dots\,,\, n\}$ has at most $n_e$ points. Calling these classes $H_1,\, H_2,\, \dots\, ,\, H_M$ with $\sum_{1\leq m\leq M} H_m  = n$, we have then $H_m\leq n_e$ for each $1\leq m\leq M$. For every $h$ large, arguing as before it is convenient to write, for each $1\leq m\leq M$,
\[
\pi^h_m = \frac 1{H_m} \, \sum_{i=H_1+\cdots + H_{m-1}+1}^{H_1+\cdots + H_m} \delta_{x^h_i}\,,
\]
so that
\[
\pi^h = \sum_{m=1}^M \frac{H_m}n\,\pi^h_m\,.
\]
By definition, any two points $x^h_i$ and $x^h_j$ so that $i$ and $j$ belong to two different classes have distance at least $d^h$, hence an interaction of at most $\eta^h$. But then, evaluating the energy as in~(\ref{firsttry}), we get
\[
\E_n(\pi^h)\geq \sum_{m=1}^M \frac{H_m^2}{n^2} \, \E_{H_m} (\pi^h_m) - n(n-1)\eta^h
\geq \sum_{m=1}^M \frac{H_m^2}{n^2} \,\inf g - n(n-1)\eta^h\,,
\]
which, keeping in mind that $\inf g<0$ by~\ref{ass:negativemeas} and letting $h\to\infty$, implies
\[
m_n = \lim \E_n(\pi^h)\geq \liminf\sum_{m=1}^M \frac{H_m^2}{n^2}\, \inf g
\geq \lim \sum_{m=1}^M \frac{H_m n_e}{n^2} \,\inf g
= \frac{n_e}n \,\inf g\,.
\]
Since this holds true for any $n>n_e$, letting $n\to \infty$ we obtain that $\liminf m_n\geq 0$, which is impossible since by Proposition~\ref{prop:limit-particles} we know that $\lim_{n\to\infty} m_n = \inf \{\E(\mu)\}<0$. The contradiction concludes the proof.
\end{proof}

The next lemma shows that if $g$ is continuous and decreasing in a neighborhood of $0$, then the potential corresponding to a minimizer of $\E$ is continuous in the convex hull of the support of the minimizer.

\begin{lemma}\label{lemma:potential-out-of-convex-spt}
Suppose that $g$ satisfies~\emph{\ref{ass:bddbelow_lsc}--\ref{ass:negativemeas}}, and let $\mu\in\P(\R^N)$ be any minimizer of $\E$. For every $\eps>0$ there exists $\rho>0$ such that, for any $\bar x\in\spt\mu\cap \partial(\convexEnv(\spt \mu))$ and any $x\in B_{\rho}(\bar x)\setminus \convexEnv(\spt \mu)$ so that $|x-\bar x| = d(x,\spt \mu)$, we have
\[
\psi_\mu(x)-\psi_\mu(\bar x)< \eps\,.
\]
\end{lemma}
\begin{proof}
The approach here is very similar to the proof of~\cite[Lemma~3.7]{CP2025}. Let $S\defeq \spt(\mu)$, that is compact thanks to Lemma~\ref{lemma:compact-support-min}, let $R \defeq \diam (S)+\bar r$, where $\bar r$ is the constant in~\ref{ass:new}, and let $\omega$ be the modulus of continuity of $g$ in $\overline{B}_R\setminus B_{\bar r/2}$. We define $\rho<\bar r/2$ as any number such that $\omega(\rho)<\eps$, and we take two points $x$ and $\bar x$ as in the claim. Let us take any $z\in S$; we claim that
\begin{equation}\label{psicont}
g(z-x) <g(z-\bar x) +\eps\,.
\end{equation}
If both $|z-x|$ and $|z-\bar x|$ are greater than $\bar r/2$, then this is true because
\[
\big|g(z-x) - g(z-\bar x)\big|\leq \omega\big(|x-\bar x|\big)\leq \omega(\rho)<\eps\,.
\]
Instead, if one among $|z-x|$ and $|z-\bar x|$ is less than $\bar r/2$, then also the other one is less than $\bar r/2+|x-\bar x|\leq \bar r/2 + \rho<\bar r$, and since by definition $|z-\bar x|\leq |z-x|$ then $g(z-x)\leq g(z-\bar x)$. Thus, in both cases~(\ref{psicont}) is true. As a consequence, to get the thesis we just have to evaluate
\[
\psi_\mu(x)-\psi_\mu(\bar x)=\int g(z-x) - g(z-\bar x) \dd\mu(z)< \eps \mu(\R^N) = \eps\,.
\]
\end{proof}

We are finally ready to show the existence of minimizers for the discretized energies $\E_n$ whenever the discretization parameter $n$ is large enough, using the result of Lemma~\ref{lemma:potential-out-of-convex-spt}. By Proposition~\ref{prop:limit-particles}, we have that a sequence of asymptotically optimal discrete energies is precompact (up to translations), but this does not guarantee that there exists a minimizing $n$-discrete probability measure for every $n$ large enough. In fact, for every fixed $n\in\N$ it could be convenient to consider a discrete measure with a small fraction of the particles that go to infinity. We will ultimately show that this is not possible because the interaction energy at infinite distance is (at least) zero, while the average interaction energy of a candidate minimizer is negative thanks to~\ref{ass:negativemeas}. Thus it is convenient to ``bring back'' the particles which are far away, to be closer to the large portion of the cluster that is converging the in weak-$*$ topology.

Before presenting our last result, let us introduce a little notation. For any $n\in\N$, consider two measures $\nu_1,\, \nu_2$ of the form
\begin{align*}
\nu_1 = \frac 1n \sum_{i=1}^{n_1} \delta_{P_i}\,, &&
\nu_2 = \frac 1n \sum_{j=1}^{n_2} \delta_{Q_j}\,,
\end{align*}
where $n_1,\, n_2\leq n$ and $P_i$ and $Q_j$ are some points in $\R^N$. In particular, $\nu_1$ and $\nu_2$ are $n$-discrete measures only if $n_1$ and $n_2$ coincide with $n$. With a small abuse of notation, we will write
\begin{equation}\label{notation1}
\E_n(\nu_1,\nu_2) = \frac 1{n^2}\sum_{i=1}^{n_1} \sum_{j=1}^{n_2} g(P_i-Q_j)\,,
\end{equation}
and we will also write
\begin{equation}\label{notation2}
\E_n(\nu_1) = \frac 1{n^2} \sum_{1\leq i\neq j\leq n_1} g(P_i-P_j) = \frac{n_1^2}{n^2}\,\E_{n_1}(\nu_1)\,.
\end{equation}
Notice that $\E_n(\nu_1)$ does not coincide with $\E_n(\nu_1,\nu_1)$, because we do not consider self-interaction of a point with itself; in particular, if $g(0)=+\infty$ (as it happens in many of the main examples), then $\E_n(\nu_1,\nu_1)$ is always $+\infty$. This is different from the continuous interaction functional, for which we use the notation $\E(\mu)=\E(\mu,\mu)$. This notation introduced in \eqref{notation1} and \eqref{notation2} will be useful in our next construction.

\begin{theorem}\label{thm:existence}
Suppose that $g$ satisfies~\emph{\ref{ass:bddbelow_lsc}--\ref{ass:negativemeas}}. Then there exists $\bar n\in\N$ such that, for every $n>\bar n$, the functional $\E_n$ admits a minimizer.
\end{theorem}
\begin{proof}
Let us assume that the claim is false. Then, there exists a sequence $n_k\nearrow +\infty$ such that every functional $\E_{n_k}$ does not admit a minimizer. For every $k$, we can select a $n_k$-discrete measure $\pi_k$, concentrated on the points $x^k_1,\, x^k_2,\, \dots\,,\, x^k_{n_k}$, such that
\begin{equation}\label{almostminimizer}
\E_{n_k} ( \pi_k) < m_{n_k} + \frac 1{k n_k}\,.
\end{equation}
Since there is no minimizer for $\E_{n_k}$, we can assume that
\begin{equation}\label{spread}
\max \Big\{ \big|x^k_i-x^k_j\big| \colon 1\leq i,\,j\leq n_k\Big\}>k\,.
\end{equation}

By Proposition~\ref{prop:limit-particles}, up to a subsequence and up to translations, the sequence $\pi_k$ weakly* converges to an optimal measure $\mu$, which by Lemma~\ref{lemma:compact-support-min} has compact support, so $\spt\mu\comp B_R$ for some radius $R$. By~\ref{ass:negativemeas}, we have $\alpha\defeq\E(\mu)<0$, so by~\ref{ass:zerolimit} there is some $R'$ such that
\begin{equation}\label{R'islarge}
g(v)\geq \frac \alpha 4\qquad \forall\, v\in\R^N,\, |v|\geq R'\,.
\end{equation}
Up to renumbering the points, for every $k$ there is a number $0\leq a_k\leq n_k$ such that
\[
|x^k_i| \leq R+R' \quad \Longleftrightarrow \quad 1\leq i \leq n_k-a_k\,.
\]
In words, the first $n_k-a_k$ points are in the large ball $B_{R+R'}$, and the last $a_k$ are outside of this ball. By~(\ref{spread}), we have $a_k\geq 1$ for every $k>2(R+R')$; that is, for every large $k$ there are some points in the support of $\pi_k$ which are very far away. However, since $\pi_k\weakstar \mu$, the mass of $\pi_k$ inside the ball $B_{R+R'}$ converges to $1$, so
\begin{equation}\label{aknkto0}
\lim_{k\to\infty} \frac{a_k}{n_k}=0\,.
\end{equation}
Let us now write for brevity $\pi_k=\pi'_k +\pi''_k$, where
\begin{align*}
\pi'_k = \frac 1{n_k} \sum_{i=1}^{n_k-a_k} \delta_{x^k_i}\,, &&
\pi''_k = \frac 1{n_k} \sum_{i=n_k-a_k+1}^{n_k} \delta_{x^k_i}\,.
\end{align*}
Basically, the measure $\pi'_k$ represents the part of $\pi_k$ which is concentrated in the ball $B_{R+R'}$. Thanks to~(\ref{aknkto0}), we also have that $\pi'_k\weakstar \mu$. Using the notation~(\ref{notation1})--(\ref{notation2}), we have
\begin{equation}\label{exactenergy}
\E_{n_k}(\pi_k) = \E_{n_k}(\pi'_k)+\E_{n_k}(\pi_k'') + 2 \E_{n_k}(\pi_k',\pi_k'')\,.
\end{equation}
Notice now that
\[
\E_{n_k}(\pi_k'') = \frac 1{n_k^2} \sum_{1\leq i\neq j\leq a_k} g(x^k_{n_k-a_k+i}-x^k_{n_k-a_k+j}) \geq \inf g\,  \frac{a_k^2}{n_k^2}\,.
\]
In order to evaluate $\E_{n_k}(\pi_k',\pi_k'')$, we call $a_k'$ the number of points in the support of $\pi_k'$ which are in the annulus $B_{R+R'}\setminus B_R$. Then, also taking into account~(\ref{R'islarge}), we have
\[\begin{split}
\E_{n_k}(\pi_k',\pi_k'') &= \frac 1{n_k^2} \sum_{i=1}^{n_k-a_k} \sum_{j=n_k-a_k+1}^{n_k} g(x^k_i-x^k_j)
\geq \frac{a_k(n_k-a_k-a_k')}{n_k^2} \cdot\frac \alpha 4 +\frac{a_k a_k'}{n_k^2} \,\inf g\\
&\geq \frac{a_k}{n_k} \cdot \frac \alpha 4 +\frac{a_k a_k'}{n_k^2} \,\inf g\,.
\end{split}\]
Again keeping in mind that $\pi_k\weakstar\mu$, together with~(\ref{aknkto0}) we know that also $a_k'/n_k\to 0$. As a consequence, inserting the last two estimates in~(\ref{exactenergy}) we obtain
\begin{equation}\label{originalcase}
\liminf_{k\to\infty}\, \Big(\E_{n_k}(\pi_k) - \E_{n_k}(\pi'_k)\Big) \cdot \frac{n_k}{a_k} \geq \frac \alpha 2 \,.
\end{equation}
This estimate says that the effect that the last $a_k$ points have to the energy of $\pi_k$ is extremely scarce (or even positive), also with respect to their total mass $a_k/n_k$. We will find a contradiction by showing that we can replace these points with other $a_k$ points, much closer to the origin, having a greater effect on the energy. Since by~(\ref{almostminimizer}) the measures $\pi_k$ are almost optimal, this will provide the desired contradiction.\par

In order to do this, we start by working on the potentials $\psi_{\pi_k'}$ and $\psi_\mu$ corresponding to the measures $\pi_k'$ and $\mu$. Notice that these potentials belong to $L^1_{\rm loc}$ by definition. We claim that
\begin{equation}\label{weaklimit}
\psi_{\pi_k'} \weakstar \psi_\mu\,.
\end{equation}
To show this property, we fix any $\phi \in C_c(\R^N)$, and we call $K=\spt\phi$. We have then
\[
\int \phi(x) \psi_{\pi_k'}(x)\dd x = \int \phi(x) \bigg( \int g(x-y)\dd\pi_k'(y) \bigg)\dd x
= \int  \bigg( \int \phi(x) g(x-y)\dd x  \bigg)\dd\pi_k'(y)\,.
\]
Now, notice that $\phi\ast g$ is a continuous and bounded function in $K+B_{R+R'}$, because $g$ is an $L^1_{\rm loc}$ function and $\phi\in C_c(\R^N)$. As a consequence, since we have already observed that $\pi_k'\weakstar \mu$, we can pass to the limit obtaining then
\[
\lim_{k\to\infty} \int \phi(x) \psi_{\pi_k'}(x)\dd x = \int  \bigg( \int \phi(x) g(x-y)  \dd x  \bigg)\dd\mu(y) = \int \phi(x) \psi_\mu(x)\dd x \,.
\]
We then obtain~(\ref{weaklimit}). Notice that it has been fundamental to consider the sequence of measures $\pi_k'$, which are all concentrated in the ball $B_{R+R'}$. The same argument could have not been done with the original sequence $\pi_k$, because $\phi\ast g$ ist not necessarily continuous and bounded in the whole $\R^N$, and then it is not possible to prove that $\psi_{\pi_k}$ converges to $\psi_\mu$.

Let us now take an extremal point $\bar x$ of $\convexEnv(\spt\mu)$, for instance let $\bar x$ be a point of $\spt\mu$ maximizing the distance from the origin. Since we know that $\psi_\mu(\bar x)\leq\alpha$ by Lemma~\ref{lemma:EL-E}, by Lemma~\ref{lemma:potential-out-of-convex-spt} we deduce that $\psi_\mu\big((1+t)\bar x\big)\leq \frac 34\, \alpha$ for some small $t>0$. Since the potential $\psi_\mu$ is continuous outside of $\spt\mu$, and the point $(1+t)\bar x$ is by construction outside $\spt\mu$, there is a small cube $Q$ of side $\eta$ centered at that point such that $\psi_\mu(x)\leq \frac 23\,\alpha$ for every $x\in Q$. We can assume that $\diam Q<\bar r$, where $\bar r$ is the constant in~\ref{ass:new}.

Let us then consider a given $k$. We define $l=\lceil\sqrt[N]{a_k}\ \rceil$, and we subdivide the cube $Q$ in $l^N$ cubes of side $\eta/l$ each. If $k$ is large enough, we can take a regular grid of points $y^k_j$, with $1\leq j \leq l^N$, each one being in one of the $l^N$ cubes, in such a way that
\[
\frac 1{l^N} \sum_{j=1}^{l^N} \psi_{\pi_k'}(y^k_j) \leq \frac \alpha 2\,.
\]
Notice that this is possible by the fact that $\psi_\mu\leq \frac 23\,\alpha$ in the whole cube $Q$, and since by~(\ref{weaklimit}) we know that $\psi_{\pi_k'}\weakstar \psi_\mu$. 
Up to renumbering the points, and keeping in mind that $l^N\geq a_k$, we can also assume that
\begin{equation}\label{goodchoice}
\frac 1{a_k} \sum_{j=1}^{a_k} \psi_{\pi_k'}(y^k_j) \leq \frac \alpha 2\,.
\end{equation}
We can then define a competitor to $\pi_k$ as $\tilde\pi_k=\pi_k'+\tilde\pi_k''$, where
\[
\tilde\pi_k''=\frac 1{n_k}\sum_{i=1}^{a_k} \delta_{y^k_i}\,.
\]
As in~(\ref{exactenergy}), we have
\begin{equation}\label{exactenergytilde}
\E_{n_k}(\tilde\pi_k) = \E_{n_k}(\pi'_k)+\E_{n_k}(\tilde\pi_k'') + 2 \E_{n_k}(\pi_k',\tilde\pi_k'')\,.
\end{equation}
By definition, and keeping in mind~(\ref{goodchoice}), we can evaluate
\begin{equation}\label{bra1}
\E_{n_k}(\pi_k',\tilde\pi_k'')= \frac 1{n_k^2}\, \sum_{i=1}^{n_k-a_k} \sum_{j=1}^{a_k} g(x^k_i-y^k_j)
=\frac 1{n_k} \,\sum_{j=1}^{a_k} \psi_{\pi_k'}(y^k_j)
\leq\frac {a_k}{n_k} \, \frac\alpha 2\,.
\end{equation}
In order to estimate $\E_{n_k}(\tilde\pi_k'')$, let us fix a generic $1\leq j\leq a_k$. We use the following elementary fact: there exists a dimensional constant $C_N>0$ such that, for any $\rho>0$ the number of points in $\{w\in\Z^N: 0<|z|<\rho\}$ is less than $C_N\rho^N$. Then, keeping in mind that the points $y^k_i$ with $1\leq i \leq l^N$ are on a grid with side-length $\eta/l$, using~\ref{ass:new} and the fact that $\diam Q\leq \bar r$, we deduce that
\[
\sum_{1\leq i \leq a_k,\, i\neq j} g(y^k_i-y^k_j) \leq C_N l^N \intmed_{[-\eta/2,\eta/2]^N} g(z)\dd z\,.
\]
Since we know that $a_k\geq 1$, we have $l^N\leq 2^{N-1} a_k$, so adding the above inequality for all $1\leq j\leq a_k$ we obtain
\begin{equation}\label{bra2}
\E_{n_k}(\tilde\pi_k'') = \frac 1{n_k^2} \sum_{j=1}^{a_k} \sum_{1\leq i \leq a_k,\, i\neq j} g(y^k_i-y^k_j)
\leq \frac{2^{N-1} C_N  a_k^2}{n_k^2} \, \intmed_{[-\eta/2,\eta/2]^N} g(z)\dd z\,.
\end{equation}
Inserting~(\ref{bra1}) and~(\ref{bra2}) into~(\ref{exactenergytilde}), and keeping in mind that $\eta$ is small but fixed and $g\in L^1_{\rm loc}$ while $a_k/n_k\to 0$ by~(\ref{aknkto0}), we get
\[
\limsup_{k\to\infty}\, \Big(\E_{n_k}(\tilde\pi_k) - \E_{n_k}(\pi'_k)\Big) \cdot \frac{n_k}{a_k} \leq \alpha \,.
\]
Coupling this estimate with~(\ref{originalcase}), also keeping in mind that $a_k\geq 1$, ensures that for $k$ large
\[
\E_{n_k}(\tilde\pi_k)\leq \E_{n_k}(\pi_k)+\frac 13\, \alpha\, \frac{a_k}{n_k}\leq \E_{n_k}(\pi_k) + \frac \alpha{3n_k}\,.
\]
However, since $\tilde\pi_k$ is a $n_k$-discrete measure and by using~(\ref{almostminimizer}) we obtain that for every $k$ large
\[
m_{n_k} \leq \E_{n_k}(\tilde\pi_k)\leq \E_{n_k}(\pi_k) - \frac \alpha{3n_k} \leq m_{n_k} + \frac 1{k n_k} +\frac \alpha{3n_k}\,,
\]
which is impossible for $k>(-\alpha)^{-1}$. The contradiction shows the thesis.
\end{proof}

\begin{remark}\label{rem:trivial-existence}
It is easy to prove the existence of minimizers for $\E_n$ for every $n\in\N$ when $g$ approaches zero from below at infinity. In fact, in this case it is not convenient to send some mass at infinity since at some finite (but maybe large) distance the interaction is strictly negative (see~\cite[Section 1.2]{BlaLew15}). In a similar way, it is easy to prove the existence of minimizers of $\E_n$ for every $n\in\N$ when $g\to+\infty$ at infinity (see~\cite[Theorem 4.1]{CanPat2018}).
\end{remark}


\bigskip
\subsection*{Acknowledgments}
This research was funded in part by the Austrian Science Fund (FWF) [grant DOI \href{https://www.fwf.ac.at/en/research-radar/10.55776/EFP6}{10.55776/EFP6}]. IT was partially supported by the Simons Foundation (851065), by the NSF-DMS (2306962), and by the Alexander von Humboldt Foundation. For open access purposes, the authors have applied a CC BY public copyright license to any author-accepted manuscript version arising from this submission.

\bibliographystyle{amsalpha}
\bibliography{references}

\end{document}